\documentclass[11pt]{amsart}

\usepackage[usenames,dvipsnames,svgnames,table]{xcolor}
\usepackage[colorlinks=true, pdfstartview=FitV, linkcolor=blue, citecolor=blue, urlcolor=blue]{hyperref}

\usepackage{geometry}                
\usepackage{graphicx}
\usepackage{amssymb}
\usepackage{epstopdf}
\usepackage{lscape}
\usepackage[utf8]{inputenc}
\usepackage{tikz,caption}
\usepackage{enumitem}
\setlist[itemize]{leftmargin=2em}
\setlist[enumerate]{leftmargin=2em}
\usepackage{booktabs}
\usepackage{multirow}
\usepackage{mathtools}
\usepackage[linesnumbered,ruled]{algorithm2e}

\definecolor{darkblue}{rgb}{0.0,0,0.7} 
\definecolor{darkred}{rgb}{0.7,0,0} 
\definecolor{darkgreen}{rgb}{0, .6, 0} 

\newcommand{\defncolor}{\color{darkred}}
\newcommand{\defn}[1]{{\defncolor\emph{#1}}} 

\newtheorem{theorem}{Theorem}[section]
\newtheorem{prop}[theorem]{Proposition}
\newtheorem{cor}[theorem]{Corollary}
\newtheorem{lemma}[theorem]{Lemma}

\newtheorem{problem}[theorem]{Problem}
\theoremstyle{definition}
\newtheorem{definition}[theorem]{Definition}
\newtheorem{example}[theorem]{Example}
\newtheorem{remark}[theorem]{Remark}
\numberwithin{equation}{section}

\newcommand{\idiot}[1]{\vspace{5 mm}\par \noindent
\marginpar{\textsc{Note}}
\framebox{\begin{minipage}[c]{0.95 \textwidth}
#1 \end{minipage}}\vspace{5 mm}\par}

\renewcommand{\idiot}[1]{}

\def\QQ{{\mathbb Q}}

\def\NN{{\mathbb N}}
\def\CC{{\mathbb C}}

\def\KK{{\mathbb K}}

\usepackage{ytableau}
\ytableausetup{smalltableaux, aligntableaux=center, textmode}

\def\unprotectedboldentry#1{\textcolor{Red}{\large{#1}}}
\def\boldentry{\protect\unprotectedboldentry}
\usepackage{tikz}
\usetikzlibrary{calc, positioning}
\usepackage{ifthen}
\newcommand{\tikztableauinternal}[1]{
    \def\newtableau{#1}
    \coordinate (x) at (-0.5,0.5);
    \coordinate (y) at (-0.5,0.5);
    \foreach \row in \newtableau {
        \foreach \entry in \row {
            \ifthenelse{\equal{\entry}{X} \OR \equal{\entry}{None}}
               {
                \node (y) at ($(y) + (1,0)$) {};
                \fill[color=gray!30] ($(y)-(0.5,0.5)$) rectangle +(1,1);
                \draw[color=gray, dotted] ($(y)-(0.5,0.5)$) rectangle +(1,1);
               }
               {
                \ifthenelse{\equal{\entry}{\boldentry X}}
                   {
                    \node (y) at ($(y) + (1,0)$) {};
                    \fill[color=gray] ($(y)-(0.5,0.5)$) rectangle +(1,1);
                    \draw ($(y)-(0.5,0.5)$) rectangle +(1,1);
                   }
                   {
                    \node (y) at ($(y) + (1,0)$) {\entry};
                    \draw ($(y)-(0.5,0.5)$) rectangle +(1,1);
                   }
               }
            }
        \coordinate (x) at ($(x)-(0,1)$);
        \coordinate (y) at (x);
        }
}

\newcommand{\tikztableau}[2][scale=0.6,every node/.style={font=\small}]{%
    \begin{tikzpicture}[#1]%
        \tikztableauinternal{#2}%
    \end{tikzpicture}%
}

\newcommand{\tikztableauscriptsize}[1]{\tikztableau[scale=0.30,every node/.style={font=\scriptsize}]{#1}}

\usepackage[colorinlistoftodos]{todonotes}

\newdimen\squaresize \squaresize=10pt
\newdimen\thickness \thickness=0.4pt

\def\square#1{\hbox{\vrule width \thickness
     \vbox to \squaresize{\hrule height \thickness\vss
        \hbox to \squaresize{\hss#1\hss}
     \vss\hrule height\thickness}
\unskip\vrule width \thickness}
\kern-\thickness}

\def\vsquare#1{\vbox{\square{$#1$}}\kern-\thickness}

\def\young#1{
\vbox{\smallskip\offinterlineskip
\halign{&\vsquare{##}\cr #1}}}

\def\thisbox#1{\kern-.09ex\fbox{#1}}
\def\downbox#1{\lower1.200em\hbox{#1}}
\newdimen\Squaresize \Squaresize=20pt
\newdimen\Thickness \Thickness=0.4pt

\def\Square#1{\hbox{\vrule width \Thickness
     \vbox to \Squaresize{\hrule height \Thickness\vss
        \hbox to \Squaresize{\hss#1\hss}
     \vss\hrule height\Thickness}
\unskip\vrule width \Thickness}
\kern-\Thickness}

\def\Vsquare#1{\vbox{\Square{$#1$}}\kern-\Thickness}

\title[Plethysm coefficients]{The mystery of plethysm coefficients}

\author[Colmenarejo]{Laura Colmenarejo}
\address[L. Colmenarejo]{Department of Mathematics, North Carolina State University, 2311 Stinson Drive, Raleigh, NC 27607, U.S.A}
\email{laura.colmenarejo.hernando@gmail.com}
\urladdr{https://sites.google.com/view/l-colmenarejo/home}

\author[Orellana]{Rosa Orellana}
\address[R. Orellana]{Mathematics Department, Dartmouth College, 6188 Kemeny Hall,
Hanover, NH 03755, U.S.A.}
\email{Rosa.C.Orellana@dartmouth.edu}
\urladdr{https://math.dartmouth.edu/~orellana/}

\author[Saliola]{Franco Saliola}
\address[F. Saliola]{D\'epartement de math\'ematiques,
Universit\'e du Qu\'ebec \`a Montr\'eal, Canada}
\email{saliola.franco@uqam.ca}
\urladdr{http://lacim.uqam.ca/~saliola/}

\author[Schilling]{Anne Schilling}
\address[A. Schilling]{Department of Mathematics, University of California, One Shields
Avenue, Davis, CA 95616-8633, U.S.A.}
\email{anne@math.ucdavis.edu}
\urladdr{http://www.math.ucdavis.edu/\~{}anne}

\author[Zabrocki]{Mike Zabrocki}
\address[M. Zabrocki]{Department of Mathematics and Statistics,  York University, 4700 Keele Street, Toronto, 
Ontario M3J 1P3, Canada}
\email{zabrocki@mathstat.yorku.ca}
\urladdr{http://garsia.math.yorku.ca/~zabrocki/}

\begin{document}

\begin{abstract}
Composing two representations of the general linear groups gives rise to Littlewood's
(outer) plethysm. On the level of characters, this poses the question of finding the Schur
expansion of the plethysm of two Schur functions. A combinatorial interpretation for the Schur 
expansion coefficients of the plethysm of two Schur functions is, in general, still
an open problem. We identify a proof technique of combinatorial representation
theory, which we call the ``$s$-perp trick'', and point out several examples in the
literature where this idea is used. We use the $s$-perp trick to give algorithms for computing 
monomial and Schur expansions of symmetric functions. In several special cases,
these algorithms are more efficient than those currently implemented in {\sc
SageMath}.
\end{abstract}

\maketitle

\section{Introduction}
\label{Introduction}

The isomorphism classes of complex irreducible polynomial representations of $GL_n := GL_n(\CC)$ are
indexed by integer partitions $\lambda$ with at most $n$ parts. We denote such a representation by $\rho^\lambda$.
Its character is identified with the \defn{Schur polynomial} (see \cite{ec2})
\begin{equation*}
        s_\lambda(x_1,\ldots,x_n) = \sum_{T \in \mathsf{SSYT(\lambda)}} x^{\mathrm{weight}(T)} ,
\end{equation*}
where $\mathsf{SSYT}(\lambda)$ is the set of all semistandard Young tableaux of shape $\lambda$ over the
alphabet $\{1,2,\ldots,n\}$ and $\mathrm{weight}(T)$ is an $n$-dimensional vector, where the $i$-th entry contains the number of 
occurrences of the letter $i$ in $T$, and where $x^\alpha = x_1^{\alpha_1} x_2^{\alpha_2} \cdots x_n^{\alpha_n}$ for any $n$-dimensional vector $\alpha$.
The composition of two such representations, say $\rho^\lambda : GL_n \to GL_m$
and $\rho^\mu : GL_m \to GL_\ell$, is also a polynomial representation of $GL_n$,
and its character is denoted by $s_\lambda[ s_\mu]$.
This operation can be viewed as an operation on symmetric polynomials, which was named \defn{(outer) plethysm} by
Littlewood~\cite{Littlewood.1944}.

The main objective of this paper is to discuss the following open problem.
\begin{problem}
\label{problem.plethysm}
Since $s_\lambda[s_\mu]$ is the character of a $GL_n$-representation, it is an $\NN$-linear combination of Schur
polynomials. Find a combinatorial interpretation of the coefficients $a_{\lambda,\mu}^\nu \in \NN$ in the expansion
\begin{equation*}
        s_\lambda[s_\mu] = \sum_\nu a_{\lambda, \mu}^\nu s_\nu.
\end{equation*}
\end{problem}
In the last century, the problem of understanding the coefficients $a_{\lambda,\mu}^\nu$ has stood as a measure
of progress in the field of $GL_n$-representation theory (see for instance Problem 9 of \cite{Stan}).
Here we identify the \defn{$s$-perp trick} as a possible way to approach this problem
by proving known cases in a simple way and finding some new combinatorial descriptions
of the $a_{\lambda,\mu}^\nu$. In addition,
we demonstrate that the $s$-perp trick gives an efficient way to compute plethysm coefficients.

The simplest form of the problem occurs when the partitions $\lambda$ and $\mu$
are both of row shape, however even this case is notoriously difficult and
explicit formulae are known only in very special circumstances. The following
table presents a (non-exhaustive) list of some of the known results in this
direction.
\begin{center}
    \begin{tabular}{ll}
        \toprule
        Formulae for plethysms of the form $s_m[s_n]$                   & References                              \\ \midrule
        $s_2[s_n]$  in terms of Schur functions                         & \cite{Littlewood1940}                   \\
        $s_3[s_n]$  in terms of Schur functions                         & \cite{Thrall1942}                       \\
        $s_4[s_n]$  in terms of Schur functions                         & \cite{Foulkes1954, Howe1987}            \\
        $s_2[s_\lambda]$ in terms of Schur functions                    & \cite{CarreLeclerc1995, vanLeeuwen2000} \\
        $s_\lambda[s_\mu]$  in terms of fundamental quasisymmetric functions    & \cite{LoehrWarrington2012}              \\
        $s_2[s_b[s_a]]$ and $s_c[s_2[s_a]]$ in terms of Schur functions & \cite{GutierrezRosas2022}               \\
        \bottomrule
    \end{tabular}
\end{center}
Since general formulae for the coefficients $a_{\lambda, \mu}^{\nu}$ have been
elusive, various methods for computing plethysm have been developed
\cite{CarbonaraRemmelYang1995, CariniRemmel1998, ChenGarsiaRemmel1984,
Colmenarejo2017, Duncan1952, Duncan1954, KahleMichalek2016, LangleyRemmel2004,
      RemmelShimozono1998, Todd1949, Wildon2018, Yang1998}
as well as representation-theoretic approaches
\cite{BowmanPaget, deBoechPagetWildon2021, Howe1987, UBP2021,
    PagetWildon2021, ScharfThibon1994}.

The paper is organized as follows.
In Section~\ref{section.notation}, we set up notation in the framework of symmetric functions and recall the definition of plethysm.
In Section~\ref{section.s perp}, we describe the $s$-perp trick to prove symmetric function
identities and identify places in the literature where this trick has been used.
In Section \ref{section.schur}, we state an algorithm for computing the
Schur expansion for a symmetric function $f$ given the Schur
expansions of $s_r^\perp f$. In Section~\ref{section.plethysm}, we
apply this algorithm to plethysm expressions to show
how it is used to speed up calculations of this type and to prove/derive new combinatorial
formulas for the coefficients $a_{\lambda, \mu}^{\nu}$ in Problem~\ref{problem.plethysm}.

\subsection*{Acknowledgments}
The authors would like to thank AIM for the opportunity to collaborate through their SQuaRE program.

The first author was partially supported by FMQ333 and PID2020-117843GB-I00. 
The second author was partially supported by NSF grant DMS-2153998 and the fourth author 
was partially supported by NSF grants DMS--1760329 and DMS--2053350.  
The third and fifth authors were supported by NSERC Discovery Grants.

\section{Symmetric functions and plethysm}
\label{section.notation}

We begin by reviewing some notation in the framework of the ring of symmetric functions. We refer
 the reader to references like  \cite{Macdonald, Sagan, ec2} for more details. 

A \defn{partition} of a positive integer $n$ is a sequence of
positive integers $\lambda = (\lambda_1, \lambda_2, \ldots, \lambda_r)$ with $\lambda_1 \geqslant \lambda_2 \geqslant \cdots
\geqslant \lambda_r > 0$ whose sum $|\lambda| := \lambda_1 + \lambda_2 + \cdots + \lambda_{r}$
is $n$. We use the notation $\lambda \vdash n$ to indicate that $\lambda$ is a partition of $n\in \mathbb{N}$.  The \defn{length} of $\lambda$ 
is denoted $\ell(\lambda):=r$. We assume that the empty partition, $\lambda=()$, is the only partition of $n=0$ and its length is $0$. Also, we 
use the notation $\overline{\lambda}=(\lambda_2,\lambda_3,\dots,\lambda_{\ell(\lambda)})$ and $\lambda'$ to denote conjugate partition.

Let $\KK$ be a ring containing $\QQ$ as a subfield, such as $\CC$ or $\QQ(q,t)$.
The ring of symmetric functions is defined as
\begin{equation*}
	\Lambda := \KK[p_1, p_2, p_3, \ldots ],
\end{equation*}
where the generators $p_r$ are known as the \defn{power sums} and
the degree of $p_r$ is equal to $r$.  The subspace of symmetric functions
of degree $n$ is denoted $\Lambda_{=n}$ and is spanned by the products $p_\lambda :=
p_{\lambda_1} p_{\lambda_2} \cdots p_{\lambda_{\ell(\lambda)}}$,
where $\lambda \vdash n$. 

There are five distinguished bases for $\Lambda_{=n}$ that we will refer to in this
exposition: the \defn{power sum} $\{ p_\lambda \}_{\lambda \vdash n}$,
 \defn{complete} (or homogeneous) $\{ h_\lambda \}_{\lambda \vdash n}$,
 \defn{elementary} $\{ e_\lambda \}_{\lambda \vdash n}$,
 \defn{monomial} $\{ m_\lambda \}_{\lambda \vdash n}$ and
 \defn{Schur} $\{ s_\lambda \}_{\lambda \vdash n}$ bases. (See  \cite{Macdonald, Sagan, ec2} for the definition of these bases and the basic relations between them.) 
 
The ring of symmetric functions is endowed with a \emph{standard involution} defined by setting 
$\omega(h_\lambda) = e_\lambda$, $\omega(p_\lambda) = (-1)^{|\lambda|+\ell(\lambda)} p_\lambda$,
or $\omega(s_\lambda) = s_{\lambda'}$. It also has a scalar product, 
known as the \defn{Hall inner product}, which is defined by 
\begin{equation*}
	\left< s_\lambda, s_\mu \right> = \left< h_\lambda, m_\mu \right>
	=\begin{cases}
	1&\hbox{ if }\lambda=\mu,\\
	0&\hbox{ otherwise}.
	\end{cases}
\end{equation*}
If multiplication by a symmetric function $f$ is thought of as a linear operator,
then the linear operator which is adjoint to multiplication by $f$ with respect to this scalar
product is denote $f^\perp$. It satisfies
$\left< f^\perp(g), h \right> = \left< g, f h \right>$ for all symmetric functions
$f,g,h \in \Lambda$ and can be calculated via the formula 
\begin{equation}
\label{equation.perp}
	f^\perp(g) = \sum_{\mu} \left< g, f s_\mu \right> s_\mu.
\end{equation}
Then for $f\in \Lambda$, we can view $f^\perp$ as a (bi-)linear operator since
$(f + g)^\perp = f^\perp + g^\perp$ and $f^\perp( g + h) = f^\perp(g) + f^\perp(h)$.
For the Schur basis, $s_\lambda^\perp(s_\mu) = \sum_\nu c^\mu_{\lambda\nu} s_\nu$, where
$c^\mu_{\lambda\nu}$ is the Littlewood--Richardson coefficient.

As one of the applications of the $s$-perp trick we will need the
notion of composition of symmetric functions, or \defn{plethysm}, and
extend its use to \defn{plethystic notation} on symmetric functions.
In \cite{Macdonald}, the operation is denoted $f \circ g$ but it is convenient to
express composition using square brackets.
If $g \in \Lambda$ with $g = \sum_\lambda c_\lambda p_\lambda$ for some coefficients
$c_\lambda \in \QQ$,\footnote{The condition that the coefficients $c_\lambda$ are in $\QQ$ is important because using plethystic notation,
if the base ring contains variables (for example, if $\KK=\QQ(q,t)$), then the variables must be treated differently.} then
$$p_r[g] = \sum_\lambda c_\lambda p_{r \lambda_1} p_{r \lambda_2}
\cdots p_{r \lambda_{\ell(\lambda)}}$$
for all positive integers $r$,
and $p_\mu[g] = p_{\mu_1}[g] p_{\mu_2}[g] \cdots p_{\mu_{\ell(\mu)}}[g]$
for all partitions $\mu$.
The plethysm of $f$ and $g$, for $f \in \Lambda$ with $f = \sum_\mu c'_\mu p_{\mu}$ with $c'_\mu \in \KK$,
is defined as
$$f[g] = \sum_\mu c'_\mu p_{\mu}[g].$$

Next, \defn{plethystic notation} extends the operation of plethysm to expressions containing variables
from the base ring $\KK$. For $E:=E(x_1, x_2, x_3, \ldots) \in \KK$, we have
\[
	p_r[E(x_1, x_2, x_3, \ldots)] = E(x_1^r, x_2^r, x_3^r, \ldots)
\]
and for $f \in \Lambda$, where $f = \sum_\lambda c_\lambda p_\lambda$ (with $c_\lambda \in \KK$),
\[
	f[E] = \sum_\lambda c_\lambda p_{\lambda_1}[E] p_{\lambda_2}[E] \cdots p_{\lambda_{\ell(\lambda)}}[E]~.
\]

We will use the symbol $X$ to stand for an arbitrary
alphabet of variables $X := x_1 + x_2 + x_3 + \cdots$, but as that expression is sufficiently general
it may be replaced with any element of $\Lambda$ and the identity will hold true. In particular,
$f[X]$ is an expression equivalent to $f$ since if $X=p_1$, then by definition, $f[p_1] = f$.

For our purposes, we require the following identities which can be derived from this
definition (see \cite{Macdonald}):
\begin{align}
	f[X+t] &= \sum_{r \geqslant 0} (s_r^\perp f)[X] t^r, &
	f[X-t] &= \sum_{r \geqslant 0} (s_{1^r}^\perp f)[X] (-t)^r,\label{equation.Xpmt}\\
	f[-X] &= (-1)^{\mathsf{degree}(f)} (\omega f)[X],&
	f[t X] &= t^{\mathsf{degree}(f)} f[X],\label{equation.omegatfactor}
\end{align}
where $f \in \Lambda$ is of homogeneous
degree equal to $\mathsf{degree}(f)$ and $t$ is a variable in the base ring $\KK$.
Then for expressions $A_1, A_2, \ldots, A_k \in \Lambda$,\begin{equation}\label{equation.LR}
	f[A_1 + A_2 + \cdots + A_k] = \sum_{\nu^{(*)}}
	s_{\nu^{(1)}}[A_1] s_{\nu^{(2)}}[A_2] \cdots s_{\nu^{(k)}}[A_k]
	\left< f, s_{\nu^{(1)}} s_{\nu^{(2)}} \cdots s_{\nu^{(k)}} \right>,
\end{equation}
where the sum is over all sequences of partitions $\nu^{(*)} = (\nu^{(1)}, \nu^{(2)}, \ldots, \nu^{(k)})$
with $\sum_{i=1}^k |\nu^{(i)}| = \mathsf{degree}(f)$.

Given two symmetric functions $f,g \in \Lambda$ with known monomial expansion $f = \sum_{i\geqslant 1} x^{a^i}$,
where $a^1, a^2, \ldots$ are vectors,
the plethysm is also given by
\[
	g[f] = g(x^{a^1}, x^{a^2},\ldots).
\]
\begin{example}
Since $s_1 = x_1 + x_2 + \cdots$, it hence immediately follows that 
\[
	g[s_1] = g(x_1,x_2,\ldots)=g
\]
and since $p_n = x_1^n + x_2^n + \cdots$, it follows that if $f = \sum_{i\geqslant 1} x^{a^i}$, then
\[
	f[p_n]= f(x_1^n,x_2^n,\ldots) = \sum_{i\geqslant 1} x^{a^in} = p_n[f].
\]
\end{example}

\begin{example}
\def\TableauOfTableaux#1#2#3#4{%
\begin{tikzpicture}[every node/.style={inner sep=2pt, outer sep=0pt, font=\scriptsize}]
    \node [draw] (row0col0) at (0, 0)
        {\begin{tikzpicture}[every node/.style={draw, outer sep=0pt, font=\tiny}]
                \node [] (A) at (0, 0) {$#1$};
                \node [right=0pt of A] {$#2$};
         \end{tikzpicture}};
    \node [right=0pt of row0col0, draw] (row0col1)
        {\begin{tikzpicture}[every node/.style={draw, outer sep=0pt, font=\tiny}]
                \node [] (A) at (0, 0) {$#3$};
                \node [right=0pt of A] {$#4$};
         \end{tikzpicture}};
\end{tikzpicture}}
For a slightly more advanced example, consider
\begin{equation}
\label{equation.schur comb}
\begin{array}{rcccccccccccc}
    s_2[x_1,x_2] &=& x_1^2 &+& x_1 x_2 &+& x_2^2\\
                 & & \tikztableauscriptsize{{1,1}} & & \tikztableauscriptsize{{1,2}} & & \tikztableauscriptsize{{2,2}}
\end{array}
\end{equation}
where we indicated the semistandard Young tableau which contributes to each term.
Then the plethysm
\begin{equation*}
    \begin{aligned}
        & s_2[s_2[x_1,x_2]] = s_2[x_1^2,x_1x_2,x_2^2] \\[1ex]
        & \arraycolsep=2pt
        \begin{array}{ccccccccccccc}
            & = & {\color{darkgreen} x_1^4}  & + & {\color{darkgreen} x_1^3x_2}  & + & {\color{darkgreen}x_1^2x_2^2}  & + & {\color{darkred}x_1^2 x_2^2}  & 
            + & {\color{darkgreen}x_1 x_2^3}  & + & {\color{darkgreen}x_2^4} \\
            &   & \tikztableauscriptsize{{1,1}} & & \tikztableauscriptsize{{1,2}} & & \tikztableauscriptsize{{1,3}} & & \tikztableauscriptsize{{2,2}} & 
            & \tikztableauscriptsize{{2,3}} & & \tikztableauscriptsize{{3,3}} \\
            &   & \TableauOfTableaux{1}{1}{1}{1} & & \TableauOfTableaux{1}{1}{1}{2} & & \TableauOfTableaux{1}{1}{2}{2} & & \TableauOfTableaux{1}{2}{1}{2} & 
            & \TableauOfTableaux{1}{2}{2}{2} & & \TableauOfTableaux{2}{2}{2}{2}
        \end{array}
        \\[1ex]
        & = {\color{darkgreen}s_4[x_1,x_2]} + {\color{darkred}s_{2,2}[x_1,x_2]}.
    \end{aligned}
\end{equation*}
Note that the first row of semistandard Young tableaux consists of the usual tableaux of shape $(2)$ in the alphabet $\{1,2,3\}$ since the Schur polynomial
in the example depends on three variables. The second row of tableaux takes into account that by~\eqref{equation.schur comb} each variable itself corresponds
to a tableau already, yielding tableaux of tableaux.
\end{example}

\section{The $s$-perp trick}
\label{section.s perp}

This section identifies a technique in symmetric function theory, which we coin the ``\defn{$s$-perp trick}''. It has been used repeatedly in the
literature to establish results in combinatorial representation theory. Here we apply it to Macdonald symmetric
functions and plethysm to compute the monomial and Schur expansions of these symmetric functions.
There are various other algorithms for the computation of plethysm that have appeared in the 
literature~\cite{Agaoka, Carvalho, ChenGarsiaRemmel1984, LangleyRemmel2004, LoehrWarrington2012, Yang1998, Yang2}.
Surprisingly, this method is in many cases more efficient than the current implementation in {\sc SageMath}; in particular, the computation of 
the Schur expansion of a single Schur plethysm of the form $s_\lambda[s_m]$ or $s_\lambda[s_{1^m}]$ is usually faster
using the $s$-perp trick.

To explain the $s$-perp trick, we consider the ring of symmetric functions $\Lambda$
spanned by the Schur basis $s_\lambda$, where $\lambda$ is a partition of a non-negative integer.
The degree of the Schur function indexed by the partition $\lambda$ is given by the size of the partition.
A function $f\in \Lambda$ is of homogeneous degree $d$ if its expansion in the Schur basis only involves Schur 
functions indexed by partitions of $d$. Recall that the ring $\Lambda$ is endowed with a scalar product, where the 
Schur basis is orthonormal. Also from~\eqref{equation.perp}, recall the action of the \defn{$s$-perp operator}.

\begin{definition}
Let $\lambda$ be a partition and $f \in \Lambda$. The \defn{action of $s_\lambda^\perp$} on $f$  is defined by
\[
	s_\lambda^\perp f = \sum_{\mu} \left< f, s_\lambda s_\mu \right> s_\mu~.
\]
\end{definition}

Now, we are ready to present the technique studied in this paper.
\begin{prop}\emph{(\defn{The $s$-perp trick})}
\label{proposition.sperptrick}
Let $f$ and $g$ be two symmetric functions of homogeneous degree $d$.  If
\[
	s_r^\perp f = s_r^\perp g \qquad \text{for all $1 \leqslant r \leqslant d$,}
\]
then $f=g$. The same statement is true if $s_r^\perp$ is replaced by $s_{1^r}^\perp$, mutatis mutandi.
\end{prop}

This statement is~\cite[Proposition 6.20.1]{haglund}, where the idea is used to compute
the monomial expansion of $\nabla(e_n)$ in the Shuffle Conjecture.
This technique occurs relatively frequently in the computation of combinatorial formulas
for $q,t$-symmetric functions with labeled lattice paths (e.g. see the proof of
\cite[Theorem 5.2]{IVW} in the recent paper by A. Iraci and A. Vanden Wyngaerd).

Proposition~\ref{proposition.sperptrick} can be interpreted as a recursive method to determine a symmetric function $f$ 
by knowing the expressions $s_r^\perp f$ for each $r$ between $1$ and the degree of $f$. In fact, the proof of this proposition 
follows from the algorithm presented in Section~\ref{section.schur},
which provides a method for recovering the monomial (Equation \eqref{equation.recurrence}) and Schur
expansions of $f$ from the symmetric functions $s_r^\perp f$. 

\vspace{0.2cm}

As we mentioned at the beginning of this section,
this technique has been used repeatedly in the literature.
For instance, one manifestation of the $s$-perp trick
is to give a monomial expansion of a symmetric polynomial using the recurrence
\begin{equation*}
	f(x_1, x_2, \ldots, x_n) = \sum_{r \geqslant 0} (s_r^\perp f)(x_1, x_2, \ldots, x_{n-1}) x_n^r~.
\end{equation*}
This is the method for computing the monomial expansion of
the modified Macdonald symmetric functions ${\tilde H}_\lambda[X;q,t]$
(see \cite{GarsiaHaiman})
given by F. Bergeron and M. Haiman in \cite[Proposition~5]{BerHai}.
They provide an explicit formula for the coefficient $d^{(r)}_{\lambda\nu}$ in the expression
defined by
\[
	s_r^\perp {\tilde H}_\lambda[X;q,t] = \sum_{\nu} d^{(r)}_{\lambda\nu} {\tilde H}_{\nu}[X;q,t]~.
\]
Denoting ${\overline \mu} = (\mu_2, \mu_3, \ldots, \mu_{\ell(\mu)})$, it
follows that if we denote the coefficient of $m_\mu$ in ${\tilde H}_\lambda[X;q,t]$ by $L_{\lambda\mu}$, 
then it can be computed with the recursive formula\footnote{
This algorithm was implemented in
2015 as the default method for computing the ${\tilde H}_{\mu}[X;q,t]$ symmetric
functions in the computer algebra system {\sc SageMath} \cite{sagemath} and
there was roughly a $2\times$ speedup against the previous most efficient
method for computing these functions.}
\begin{equation}\label{equation.recurrence}
	L_{\lambda\mu} = \sum_\beta d^{(\mu_1)}_{\lambda \beta} L_{\beta{\overline \mu}}~.
\end{equation}

Another example of the $s$-perp trick used in combinatorial representation
theory appears in a paper by A. Garsia and C. Procesi~\cite{GP}, where the authors show
 that the Frobenius characteristic
of a certain quotient module is equal to the Hall--Littlewood symmetric function.
Let $H_\lambda[X;q] = \sum_\mu K_{\lambda\mu}(q) s_\mu$
be the Hall--Littlewood symmetric function with
$K_{\lambda\mu}(q)$ representing the $q$-Kostka coefficient (see \cite[Section III.6]{Macdonald})
and let $C_\lambda[X;q]$ be the Frobenius characteristic of a certain
module. It is shown in~\cite{GP} that since 
$s_{1^r}^\perp H_\lambda[X;q]$ and $s_{1^r}^\perp C_\lambda[X;q]$ satisfy
exactly the same recursive expression, the identity $H_\lambda[X;q] = C_\lambda[X;q]$ holds.

Iraci, Rhoades and Romero use an almost identical method \cite[Lemma 3.1]{IRR} to prove that
the diagonal fermionic coinvariants $\bigwedge\{ \theta_n, \xi_n \}/I^+$
have Frobenius image of a special case of the Theta conjecture \cite[Conjecture 9.1]{DIW}.

We will apply this technique to the application of computing the plethysm of symmetric functions.

\section{Application: Schur expansions}
\label{section.schur}

In the previous section, we saw that the monomial expansion
of a symmetric function $f$ can be computed by recursively computing the
monomial expansion of $s_r^\perp f$.  In this section,
we state an algorithm for computing the Schur expansion
of a symmetric function $f$ by recursively computing the
Schur expansion of $s_r^\perp f$ (respectively $s_{1^r}^\perp f$).

Let $\mathsf{addrow}(s\_expansion, r)$ be a function which takes as input
a Schur expansion of an element of $\Lambda$ and a positive integer $r$.
If each partition indexing the terms in the Schur expansion of $f$ has first part less than
or equal to $r$, then return the expression with each indexing
partition having a row of size $r$ appended to it, or $0$ otherwise.

\begin{remark}
The idea for this algorithm occurred to us because we had a symmetric group module
and we could combinatorially describe $s_1^\perp f$, where $f$ is the Frobenius
image of this module. From this we wanted to
deduce a combinatorial formula for $f$.
In general, $s_1^\perp f$ does not characterize $f$ (e.g. consider
the symmetric functions $s_{22} + s_4$ and $s_{31}$ which both
satisfy $s_1^\perp(s_{22} + s_4) = s_{21} + s_3 = s_1^\perp(s_{31})$).
However the idea seemed quite close.  Then
we asked what additional information would we need to recover $f$
and realized that there is a general technique that is regularly
used in symmetric function theory that could be identified.
\end{remark}

\begin{algorithm}
    \SetKwInOut{Input}{Input}
    \SetKwInOut{Output}{Output}

    \underline{function ExpandSchur} $A$\;
    \Input{$A$ - array with $A[r]$ the Schur expansion of $s_r^\perp f$ for $1 \leqslant r \leqslant \mathsf{degree}(f)$}
    \Output{$out$ - a Schur expansion of $f$}
    \Begin{
    $out \longleftarrow 0$;\\
    \For{$r = length(A)$ downto $1$}{
        $out \longleftarrow out + \mathsf{addrow}(A[r] - s_r^\perp(out), r)$;
    }
    return $out$;
    }
    \caption{An algorithm for finding the Schur expansion of $f$ from the Schur expansion
    of $s_r^\perp f$ for $1 \leqslant r \leqslant \mathsf{degree}(f)$}
    \label{algorithm.expandschur}
\end{algorithm}

Algorithm \ref{algorithm.expandschur} works because when $r=k$ at line $4$,
$out$ is equal to the sum of the terms of the Schur expansion of $f$ which
have first part of the indexing partition greater than or equal to $k$ and so
$f-out$ is in the linear span of Schur functions with parts less than or equal to $k$.
Hence
\[
	\mathsf{addrow}(A[k] - s_k^\perp(out),k) = \mathsf{addrow}(s_k^\perp(f - out),k)
\]
is equal to the sum of terms in the Schur expansion of $f$ with
first part of the indexing partition equal to $k$.
When the $for$ loop completes, $out$ will equal the Schur expansion of $f$.

This algorithm can be modified to deduce the Schur expansion of
$f$ from the Schur expansion of $s_{1^r}^\perp f$, for $1 \leqslant r \leqslant \mathsf{degree}(f)$,
by replacing the function $\mathsf{addrow}$ with $\mathsf{addcol}$, which adds a column
on the partitions indexing the expansions of the Schur functions.

\section{Application to plethysm}
\label{section.plethysm}

\subsection{$s$-perp formulae for plethysm}
Equation \eqref{equation.Xpmt} can be used to compute $s_r^\perp$ or $s_{1^r}^\perp$
acting on plethysms.  The expression that we will
specialize for our application is the following identity.
\begin{prop}\label{proposition.master1}
For partitions $\lambda$ and $\mu$ and a positive integer $r$,
\begin{equation}\label{equation.identityinfinity}
s_r^\perp s_\lambda[s_\mu]
=\sum_{\nu^{(i)}} s_{\nu^{(0)}}[s_\mu]
s_{\nu^{(1)}}[s_1^\perp s_\mu]
s_{\nu^{(2)}}[s_2^\perp s_\mu] \cdots
s_{\nu^{(r')}}[s_{r'}^\perp s_\mu]
\left< s_\lambda, s_{\nu^{(0)}} s_{\nu^{(1)}} \cdots s_{\nu^{(r')}}\right>
\end{equation}
where $r' = \min(r, \mu_1)$ and the sum is over all sequences of partitions
$(\nu^{(0)}, \nu^{(1)}, \ldots, \nu^{(r')})$
with $\sum_{i=0}^r |\nu^{(i)}| = |\lambda|$ and $\sum_{i=0}^r i |\nu^{(i)}| = r$.
\end{prop}

\begin{proof}
If we let $r' = \min(r, \mu_1)$, then by Equation \eqref{equation.Xpmt} we have
\[
    s_\mu[X+t] = s_\mu + t s_1^\perp s_\mu + t^2 s_2^\perp s_\mu + \cdots + t^{r'} s_{r'}^\perp s_\mu.
\]
Now Equation \eqref{equation.identityinfinity} follows from an application of
Equation \eqref{equation.LR}, factoring out a power of $t$ using
the right expression of Equation \eqref{equation.omegatfactor}
and then taking a coefficient of $t^r$ on both sides
of the equation.
\end{proof}

A nearly identical proof will derive a similar identity for
$s_{1^r}^\perp s_\lambda[s_\mu]$ by applying Equation~\eqref{equation.Xpmt}
and in addition using Equation \eqref{equation.omegatfactor}
so that in the product we have terms of the form $(\omega s_{\nu^{(k)}})[ s_{1^k}^\perp s_\mu]
= s_{(\nu^{(k)})'}[ s_{1^k}^\perp s_\mu]$
if $k$ is odd and $s_{\nu^{(k)}}[ s_{1^k}^\perp s_\mu ]$ if $k$ is even.
\begin{prop}\label{proposition.master2}
For partitions $\lambda$ and $\mu$ and a positive integer $r$,
$s_{1^r}^\perp s_\lambda[s_\mu]$ is equal to
\begin{equation}\label{equation.identityinfinitym1}
\sum_{\nu^{(i)}} s_{\nu^{(0)}}[s_\mu]
(\omega s_{\nu^{(1)}})[s_1^\perp s_\mu]
s_{\nu^{(2)}}[s_{1^2}^\perp s_\mu] \cdots
(\omega^{r'} s_{\nu^{(r')}})[s_{1^{r'}}^\perp s_\mu]
\left< s_\lambda, s_{\nu^{(0)}} s_{\nu^{(1)}} \cdots s_{\nu^{(r')}}\right>
\end{equation}
where $r' = \min(r, \ell(\mu))$ and the sum is over all sequences of partitions
$(\nu^{(0)}, \nu^{(1)}, \ldots, \nu^{(r')})$
with $\sum_{i=0}^r |\nu^{(i)}| = |\lambda|$ and $\sum_{i=0}^r i |\nu^{(i)}| = r$.
\end{prop}

Applying Algorithm \ref{algorithm.expandschur} to either
Proposition \ref{proposition.master1} or \ref{proposition.master2}
will have too many terms, and so it cannot be considered an efficient method for computing plethysms in general. However, 
special cases of these identities can be used to recursively compute
classes of plethysms much more efficiently than the standard
algorithms.

Special cases of Equations \eqref{equation.identityinfinity} and \eqref{equation.identityinfinitym1} include
\begin{align*}
    s_{r}^\perp s_\lambda[ s_{1^w} ] &= \sum_{\mu, \gamma} c^\lambda_{\mu\gamma}
    s_\mu[s_{1^w}] s_\gamma[s_{1^{w-1}}]
    &
    s_{1^r}^\perp s_\lambda[ s_w ] &= \sum_{\mu, \gamma} c^\lambda_{\mu\gamma'}
    s_\mu[s_w] s_\gamma[s_{w-1}],
\end{align*}
where the sums are over partitions $\mu$ and $\gamma$ with $|\mu| = |\lambda| - r$ and
$|\gamma| = r$ and $c^\lambda_{\mu\gamma} := \left< s_\lambda, s_\mu s_\gamma\right>$
are the Littlewood--Richardson coefficients.

While we saw improvements in recursive computations with these formulas,
even smaller classes of functions are the special cases:
\begin{align}\label{equation.colrows}
	s_r^\perp s_{1^h}[s_{1^w}] &= s_{1^{h-r}}[s_{1^w}] s_{1^r}[s_{1^{w-1}}]&
    s_{1^r}^\perp s_{h}[s_{w}] &= s_{{h-r}}[s_{w}] s_{1^r}[s_{w-1}]\\
    \label{equation.rowscol}
	s_r^\perp s_{h}[s_{1^w}] &= s_{{h-r}}[s_{1^w}] s_{r}[s_{1^{w-1}}]&
    s_{1^r}^\perp s_{1^h}[s_{w}] &= s_{1^{h-r}}[s_{w}] s_{r}[s_{w-1}]~.
\end{align}

We compared this method with the current method used by the {\sc SageMath}
computer algebra system \cite{sagemath} and found it to be significantly faster
for computing the plethysms $s_{1^h}[s_{1^w}]$. The current implementation in
{\sc SageMath} uses the definition stated in Section~\ref{section.notation},
which requires a change of basis from the Schur functions to power sums and
back again.
An implementation of our method in {\sc SageMath} was able to compute
$s_{1^4}[s_{1^6}]$ in less than a second compared to 36 seconds with the
current implementation in {\sc SageMath}; and to compute $s_{1^6}[s_{1^6}]$ it took 21
seconds compared to well over an hour.

\subsection{A formula for $s_\lambda[s_k]$ for $|\lambda|\leqslant 3$}

Quasi-polynomial and integer polytope expressions for the coefficients of $s_\mu$ in
$s_\lambda[s_k]$ with $|\lambda| \leqslant 3$ and $k$ a positive integer have been extensively studied
\cite{Agaoka2002, ChenGarsiaRemmel1984, DentSiemons2000, Howe1987, KahleMichalek2016, Plunkett1972,    
Tetreault2020, Thrall1942}.  The quasi-polynomial expressions are
probably the most efficient means of computation
of a single coefficient in this expression.  In \cite{KahleMichalek2018}, the authors use the
expressions for $s_3[s_k]$ and $s_k[s_3]$ to conclude that
certain families of these coefficients do not need to be given by
Ehrhart functions of rational polytopes.

In the following theorem we use the $s$-perp trick and some of the properties
and analysis that others have
derived about these formulae to give a combinatorial interpretation in terms of semi-standard
Young tableaux.  This formulation gives an idea of what we expect a
combinatorial interpretation for the plethysm coefficients should look like in general.
The result was obtained independently using different methods by Florence
Maas-Gari\'epy and \'Etienne T\'etreault \cite{MaasGariepyTetreault},
and our proof makes use of one of their earlier results \cite{Tetreault2020}.
Let $\mathsf{SSYT}_{(k,k,k)}$ represent the set of semi-standard Young tableaux
with $k$ $1$s, $k$ $2$s and $k$ $3$s and denote the shape
of a tableau by $\mathsf{shape}(S)$.

\def\tabsizethree{{\Large\Bigl\{}\scalebox{0.8}{\young{1&2&3\cr}},
\raisebox{-.05in}{\scalebox{0.8}{\young{3\cr1&2\cr}}},
\raisebox{-.05in}{\scalebox{0.8}{\young{2\cr1&3\cr}}},
\raisebox{-.1in}{\scalebox{0.8}{\young{3\cr2\cr1\cr}}} {\Large\Bigr\}}}

\begin{theorem}\label{th:case3}
\squaresize=11pt
Let $T \in \tabsizethree$ be a standard tableau with shape
a partition of 3. Then for any $k\geqslant 1$,
\begin{equation}\label{eq:typesum}
	s_{\mathsf{shape}(T)}[s_k] = \sum_{\substack{S \in \mathsf{SSYT}_{(k,k,k)}\\
	\mathsf{type}(S)=T}} s_{\mathsf{shape}(S)},
\end{equation}
where $\mathsf{type}(S)$ for $S \in \mathsf{SSYT}_{(k,k,k)}$ is given in Definition~\ref{def:type} below.
\end{theorem}

\begin{definition} \label{def:type}
\squaresize=11pt
For $S \in \mathsf{SSYT}_{(k,k,k)}$, let $N_r$ represent the number of cells with the label $r$ in the second row of $S$ for $1\leqslant r \leqslant 3$.
\begin{itemize}
    \item If $S$ is standard and $k=1$, then let $\mathsf{type}(S)=S$.

    \item If $S$ has one or two rows, then we define
\begin{enumerate}[label=(\alph*)]
\item $\mathsf{type}(S) = \scalebox{0.8}{\young{1&2&3\cr}}$ if $N_{2}$ is even, $N_{3} \geqslant 2 N_{2}$,
but $N_{3} \neq 2 N_{2} +1$.
\item $\mathsf{type}(S) = \scalebox{0.8}{\young{3\cr2\cr1\cr}}$ if $N_{2}$ is odd, $N_{3} \geqslant 2 N_{2}$,
but $N_{3} \neq 2 N_{2} +1$.
\item $\mathsf{type}(S) = \scalebox{0.8}{\young{3\cr1&2\cr}}$ if $N_{2}$ is even, and $N_{3} < 2 N_{2}$
or $N_{3} = 2 N_{2} +1$.
\item $\mathsf{type}(S) = \scalebox{0.8}{\young{2\cr1&3\cr}}$ if $N_{2}$ is odd, and $N_{3} < 2 N_{2}$
or $N_{3} = 2 N_{2} +1$.
\end{enumerate}

    \item
If $S$ has three rows and $\overline S$ is $S$ with the first column removed, then
$\mathsf{type}(S) = \mathsf{type}(\overline S)^t$.
\end{itemize}
For $T$ a standard tableau of size $3$, let
$$\mathsf{Tab}_{T,k} = \{ S \in \mathsf{SSYT}_{(k,k,k)} \mid \mathsf{type}(S) = T \}.$$
\end{definition}

Define a linear operator on symmetric functions as
$$s_\mu \downarrow_k = \begin{cases} s_\mu &\hbox{ if }\ell(\mu) \leqslant k,\\
0 & \hbox{ else.}
\end{cases}$$

Also define the following bilinear (and commutative) operator by
$$s_\mu \odot s_\lambda = s_{\mu + \lambda},$$
where the sum of the two partitions is done componentwise, adding zeros if necessary.

\begin{lemma} \label{lem:works123} 
For $k\geqslant 1$,
\squaresize=11pt
$$s_3[s_k] \downarrow_2 = \sum_S s_{\mathsf{shape}(S)},$$
where the sum is over all $S \in \mathsf{Tab}_{\scalebox{0.8}{\young{1&2&3\cr}},k}$ with
$\ell(S) \leqslant 2$.
\end{lemma}

\begin{proof}
A result of \cite[Theorem 2.1]{Tetreault2020} says that
$$s_3[s_k] \downarrow_2 = s_{66} \odot s_3[s_{k-4}] \downarrow_2 + \sum_{r=2}^k s_{3k-r,r} + s_{3k}.$$
\squaresize=11pt
The tableaux $S \in \mathsf{Tab}_{\scalebox{0.8}{\young{1&2&3\cr}},k}$ and $\ell(S) \leqslant 2$
also satisfies this recurrence since they either contain $\scalebox{0.8}{\young{2&2&3&3&3&3\cr}}$ in the
top row and $\scalebox{0.8}{\young{1&1&1&1&2&2\cr}}$ in the bottom row or not.  Those that do are
(by induction) enumerated by the expression $s_{66} \odot s_3[s_{k-4}] \downarrow_2$.
Those tableaux that do not contain $\scalebox{0.8}{\young{2&2&3&3&3&3\cr}}$ and $\scalebox{0.8}{\young{1&1&1&1&2&2\cr}}$
of type $\scalebox{0.8}{\young{1&2&3\cr}}$
are either a single row or have a second row consisting of
exactly $r$ $3$'s for $2 \leqslant r \leqslant k$.
\end{proof}

\begin{lemma}\label{lem:312eq213}
For $k \geqslant 1$,
\squaresize=11pt
$$\sum_{S \in \mathsf{Tab}_{\scalebox{0.8}{\young{3\cr1&2\cr}},k}} s_{\mathsf{shape}(S)}=
\sum_{S \in \mathsf{Tab}_{\scalebox{0.8}{\young{2\cr1&3\cr}},k}} s_{\mathsf{shape}(S)}~.$$
\end{lemma}

\begin{proof}
We define a map which we denote
\squaresize=11pt
\[
	\phi \colon \mathsf{Tab}_{\scalebox{0.8}{\young{3\cr1&2\cr}},k} \rightarrow \mathsf{Tab}_{\scalebox{0.8}{\young{2\cr1&3\cr}},k}~.
\]

\noindent
\textbf{Case 1:} If the length of $S$ is 3, then let $\overline S \in \mathsf{SSYT}_{(k,k,k)}$ be the tableau $S$ with 
the first column removed. Define $\phi(S)$ to be a column of length $3$ added to $\phi^{-1}({\overline S})$.

\smallskip

\noindent
\textbf{Case 2:} If the length of the second row of $S$ is odd, then let
$S'$ be the tableau where one takes a $3$ from the second row of $S$ and exchanges
it with a $2$ in the first row.

\smallskip

\noindent
\textbf{Case 3:} If the length of the second row of $S$ is even, then let
$S'$ be the tableau such that one takes a $3$ from the first row of $S$ and exchanges
it with a $2$ in the second row.

We leave to the reader to check the details that this is a bijection.
\end{proof}

\begin{proof}[Proof of Theorem~\ref{th:case3}]
\squaresize=11pt
Assume by induction that \eqref{eq:typesum} holds for some fixed $k$ and all
$T \in \tabsizethree$.
By Lemma~\ref{lem:works123} the terms of length less than or equal to
$2$ in $s_{3}[s_{k+1}]$ are given by the $S \in \mathsf{Tab}_{\scalebox{0.8}{\young{1&2&3\cr}},k+1}$
such that $\ell(S) \leqslant 2$ and by~\eqref{equation.colrows} we have
$s_{111}^\perp s_{3}[s_{k+1}] = s_{111}[s_k]$. Then, 
it follows by induction that the terms of length 3 are given by those
$S \in \mathsf{Tab}_{\scalebox{0.8}{\young{1&2&3\cr}},k+1}$ with $\ell(S)=3$.
Therefore,
\begin{equation}\label{eq:right123}
	s_3[s_{k+1}] = \sum_{S \in \mathsf{Tab}_{\scalebox{0.8}{\young{1&2&3\cr}},k+1}} s_{\mathsf{shape}(S)}~.
\end{equation}

From \cite[Section I.8 Example 9]{Macdonald},
\[
	s_2[s_{k+1}] = \sum_{S} s_{\mathsf{shape}(S)},
\]
where the sum is over $S \in \mathsf{SSYT}_{(k+1,k+1)}$ with $N_2$ is even. Therefore,
\[
	s_{21}[s_{k+1}] + s_{3}[s_{k+1}] = s_2[s_{k+1}] s_{k+1}
	= \sum_{S \in \mathsf{Tab}_{\scalebox{0.8}{\young{1&2&3\cr}}, k+1}} s_{\mathsf{shape}(S)} +
	\sum_{S \in \mathsf{Tab}_{\scalebox{0.8}{\young{3\cr1&2\cr}}, k+1}} s_{\mathsf{shape}(S)},
\]
where the last equality holds because $\mathsf{Tab}_{\scalebox{0.8}{\young{1&2&3\cr}}, k+1} \cup \mathsf{Tab}_{\scalebox{0.8}{\young{3\cr1&2\cr}}, k+1}$
are all $S \in \mathsf{SSYT}_{(k+1,k+1,k+1)}$ such that
$N_2$ is even.  We conclude by subtracting Equation~\eqref{eq:right123} from both sides of the equation
and Lemma~\ref{lem:312eq213} that
\begin{equation}\label{eq:right312and213}
	s_{21}[s_{k+1}]
	= \sum_{S \in \mathsf{Tab}_{\scalebox{0.8}{\young{3\cr1&2\cr}},k+1}} s_{\mathsf{shape}(S)}
	= \sum_{S \in \mathsf{Tab}_{\scalebox{0.8}{\young{2\cr1&3\cr}},k+1}} s_{\mathsf{shape}(S)}~.
\end{equation}

We again use~\cite[Section I.8 Example 9]{Macdonald}, which says that
\[
	s_{11}[s_{k+1}] = \sum_{S} s_{\mathsf{shape}(S)},
\]
where the sum is over $S \in \mathsf{SSYT}_{(k+1,k+1)}$ with $N_2$ is odd so that
$$s_{21}[s_{k+1}] + s_{111}[s_{k+1}]
= s_{11}[s_{k+1}] s_{k+1}
= \sum_{S \in \mathsf{Tab}_{\scalebox{0.8}{\young{3\cr2\cr1\cr}}, k+1}} s_{\mathsf{shape}(S)} +
\sum_{S \in \mathsf{Tab}_{\scalebox{0.8}{\young{2\cr1&3\cr}}, k+1}} s_{\mathsf{shape}(S)}.$$
The last equality holds because 
$\mathsf{Tab}_{\scalebox{0.8}{\young{3\cr2\cr1\cr}}, k+1} \cup \mathsf{Tab}_{\scalebox{0.8}{\young{2\cr1&3\cr}}, k+1}$
are all $S \in \mathsf{SSYT}_{(k+1,k+1,k+1)}$ such that
$N_2$ is odd.  It then follows by subtracting Equation \eqref{eq:right312and213}
from both sides of the equation that
$$
s_{111}[s_{k+1}] =
\sum_{S \in \mathsf{Tab}_{\scalebox{0.8}{\young{3\cr2\cr1\cr}}, k+1}} s_{\mathsf{shape}(S)}~.
$$
This concludes the proof by induction since
Equation~\eqref{eq:typesum} holds for $k \rightarrow k+1$
and all $T \in \tabsizethree$.
\end{proof}

\subsection{Explicit formulas for $s_\lambda[s_2]$ and $s_\lambda[s_{1^2}]$ in special cases}

A partition $\lambda$ is called \defn{even} if all columns have even length. A partition $\lambda$ is called \defn{threshold} if
$\lambda'_i = \lambda_i+1$ for all $1\leqslant i \leqslant d(\lambda)$ where $d(\lambda)$ is the maximal $i$ such that $(i,i) \in \lambda$.

\begin{theorem}
\label{theorem.row col}
We have
\begin{align}
\label{equation.row plethysm}
	s_h[s_2] &= \sum_{\substack{\lambda \vdash 2h\\ \lambda \text{ even}}} s_{\lambda'},&
	s_h[s_{1^2}] &= \sum_{\substack{\lambda \vdash 2h\\ \lambda \text{ even}}} s_{\lambda},\\
	\label{equation.col plethysm}
	s_{1^h}[s_2] &= \sum_{\substack{\lambda \vdash 2h\\ \lambda \text{ threshold}}} s_{\lambda'},&
        s_{1^h}[s_{1^2}] &= \sum_{\substack{\lambda \vdash 2h\\ \lambda \text{ threshold}}} s_{\lambda}.
\end{align}
\end{theorem}

\begin{remark}
The second equation in~\eqref{equation.col plethysm} has appeared in~\cite[Theorem 5.2]{KlivansReiner.2008}.
The formulas of Theorem~\ref{theorem.row col} also appear in~\cite[Page 138]{Macdonald} and were originally
proven in \cite[Equations (11.9; 1) through (11.9; 4)]{Littlewood1940}, however these proofs are different from the ones we present here.
\end{remark}

A \defn{cell} of a partition $\lambda$ is a pair $(r,c)$, where
$1 \leqslant r \leqslant \ell(\lambda)$ and $1 \leqslant c \leqslant \lambda_r$.
A \defn{corner} of a partition is a cell of the partition $(r,c)$ such that
$(r+1,c)$ and $(r,c+1)$ are not cells of the partition.
For the proof of Theorem~\ref{theorem.row col}, we need the notion of \defn{opposite cell} for each cell $(s,t)$ in a threshold partition 
$\lambda$. The opposite cell $\mathsf{op}(s,t)$ of $(s,t)$ is defined to be $(t+1,s)$ if $s\leqslant t$ and $(t,s-1)$ otherwise.

\begin{proof}[Proof of Theorem~\ref{theorem.row col}]
We prove the second equation in~\eqref{equation.row plethysm} by checking that 
\[
	s_r^\perp s_{h}[s_{1^2}] = s_{{h-r}}[s_{1^2}] s_{r}
\]
from~\eqref{equation.rowscol} holds for all $r$. We have 
\[
	s_r^\perp s_\lambda = \sum_\nu s_\nu,
\]
where the sum is over all $\nu$ such that $\lambda/\nu$ has at most one cell per column
and $|\lambda/\nu|=r$.
In the columns, where there is a cell in $\lambda/\nu$, the length of the columns of $\nu$
will be odd.  This implies that $\nu$ contains an even partition $\mu$ with $\nu/\mu$ containing the
cells which make those columns odd and so will also have at most one cell per column.  Therefore,
\[
	s_r^\perp s_{h}[s_{1^2}] 
	= s_r^\perp \sum_{\substack{\lambda \vdash 2h\\ \lambda \text{ even}}} s_\lambda
	= \sum_{\substack{\nu \vdash 2h-r\\ \nu/\mu \text{ horizontal $r$-strip}\\
	\text{for some } \mu \vdash 2h-2r \text{ even}}} s_\nu
	= \sum_{\substack{\mu \vdash 2h-2r\\ \mu \text{ even}}} s_\mu s_r
	= s_{h-r}[s_{1^2}] s_r,
\]
where the second to last equality follows from the Pieri rule.

Similarly, we prove the second equation in~\eqref{equation.col plethysm} by checking that
\[
	s_r^\perp s_{1^h}[s_{1^2}] = s_{1^{h-r}}[s_{1^2}] s_{1^r}
\]
from~\eqref{equation.colrows} holds for all $r$. We start with
\[
	s_r^\perp s_{1^h}[s_{1^2}] 
	= s_r^\perp \sum_{\substack{\lambda \vdash 2h\\ \lambda \text{ threshold}}} s_\lambda
	= \sum_{\substack{\nu \vdash 2h-r\\ \lambda/\nu \text{ vertical $r$-strip}\\ \text{for some $\lambda \vdash 2h$ threshold}}} s_\nu.
\]
For each cell $(s,t)$ in $\lambda/\nu$ in the last sum from rightmost to leftmost (that is largest $t$ to smallest $t$), check whether 
$\mathsf{op}(s,t)$ is in $\nu$. If it is, remove $\mathsf{op}(s,t)$ from $\nu$, otherwise leave the partition unchanged. Call the partition with
all (possible) opposite cells removed $\tau$. Note that $\nu/\tau$ is a vertical strip since the cells in $\lambda/\nu$ form a horizontal
strip and $\mathsf{op}(s,t)$ involves transposition. For each cell $(s,t)$ in $\lambda/\nu$ from largest to smallest $t$ with $t \geqslant s$, for which 
$(s',t')=\mathsf{op}(s,t)$ is not in $\nu$, find the cell $(x,y)$ in $\tau$ with smallest $y \geqslant t'$ such that $\nu/(\tau\setminus \{(x,y)\})$ is a vertical 
strip. Remove $(x,y)$ and $\mathsf{op}(x,y)$ from $\tau$. Call the resulting partition $\mu$. Note that $\mu$ is threshold and furthermore, $\nu/\mu$ is a 
vertical $r$-strip. Hence	
\[	
	s_r^\perp s_{1^h}[s_{1^2}] 
	= \sum_{\substack{\nu \vdash 2h-r\\ \nu/\mu \text{ vertical $r$-strip}\\ \text{for some $\mu\vdash 2h-2r$ threshold}}} s_\nu
	=\sum_{\substack{\mu \vdash 2h-2r\\ \mu \text{ threshold}}} s_\mu s_{1^r} = s_{1^{h-r}}[s_{1^2}] s_{1^r}.
\]

The other two formulas follow from the identity
\begin{equation}
\label{equation.s transpose}
	s_\lambda[s_{\mu'}] = \sum_\gamma \left< s_\lambda[s_\mu], s_\gamma \right> s_{\gamma'}
\end{equation}
for $|\mu|$ even (see Equation~\eqref{equation.omegatfactor}).
\end{proof}

For $h \geqslant 2$ and $1 \leqslant r < h$, we have
\begin{align}
	\label{equation.small hook}
	s_{h-r}[s_{1^2}] s_{1^r}[s_{1^2}] =  s_{(h-r,1^r)}[s_{1^2}] + s_{(h-r+1,1^{r-1})}[s_{1^2}]~,
\end{align}
since by the dual Pieri rule $s_{h-r} s_{1^r} = s_{(h-r,1^r)} + s_{(h-r+1,1^{r-1})}$.

The next result is new.

\begin{cor}
\label{cor.small hook}
We have
\begin{align}
	\label{equation.s hook 1}
	s_{(h-1,1)}[s_{1^2}] &= \sum_{\mu \in \mathcal{P}_{2h}} s_\mu + \sum_{\nu \vdash 2h \text{ even}} (b_{\nu}-1) s_\nu,\\
	\label{equation.s hook 2}
	s_{(h-1,1)}[s_2] &= \sum_{\mu \in \mathcal{P}_{2h}} s_{\mu'} + \sum_{\nu \vdash 2h \text{ even}} (b_{\nu}-1) s_{\nu'},
\end{align}
where $\mathcal{P}_{2h}$ is the set of all partitions of $2h$ with columns of even length except two columns of distinct odd length, and
$b_\nu$ is the number of corners of $\nu$.
\end{cor}

\begin{proof}
By~\eqref{equation.small hook} with $r=1$, we have $s_{(h-1,1)}[s_{1^2}] = s_{h-1}[s_{1^2}] s_{1^2} - s_h[s_{1^2}]$.
On the other hand by Theorem~\ref{theorem.row col}, the Schur expansion of $s_{h-1}[s_{1^2}]$ contains the sum over $s_\nu$ where $\nu \vdash 2h-2$ 
is even. Multiplication by $s_{1^2}$ adds a vertical strip of size 2 to the partitions in the Schur expansion by the Pieri rule, so that
\[
	s_{h-1}[s_{1^2}] s_{1^2} = \sum_{\mu \in \mathcal{P}_{2h}} s_\mu + \sum_{\nu \vdash 2h \text{ even}} b_{\nu} s_\nu.
\]
This follows from the fact that adding a vertical strip of length 2 to an even partition gives a partition with two different odd and otherwise even length 
columns (if the two boxes are added to different columns) or an even partition (if the two boxes are added to the same column). An even partition $\nu$ of 
size $2h$ can be obtained in $b_\nu$ ways from an even partition of size $2h-2$ by adding two boxes to a column. Noting that, by Theorem~\ref{theorem.row col}, the Schur expansion of $s_h[s_{1^2}]$ contains all even partitions of size $2h$ proves~\eqref{equation.s hook 1}.

Equation~\eqref{equation.s hook 2} follows from~\eqref{equation.s transpose}.
\end{proof}

\begin{remark}
Note that there is an involution on the partitions in $\mathcal{S}_{2h} = \mathcal{P}_{2h} \cup \{\nu \vdash 2h \mid \text{$\nu$ even}\}$ appearing in the 
expansion in Corollary~\ref{cor.small hook}. Namely, map the partition $\nu=(\nu_1,\nu_2,\ldots,\nu_\ell)$ with $\ell$ even (with possibly $\nu_\ell=0$) to
\[
	w \colon \nu \mapsto (\nu_1+\nu_2,\nu_3+\nu_4,\ldots,\nu_{\ell-1}+\nu_\ell)'.
\]
Note that $w(\nu)\in \mathcal{P}_{2h}$ if $\nu \in \mathcal{P}_{2h}$ and $w(\nu)$ is even if $\nu$ is even. Also, it is not hard to see that $w^2(\nu)=\nu$
and $b_{w(\nu)} = b_{\nu}$. The involution $w$ imposes a symmetry on the Schur expansion of $s_{(h-1,1)}[s_{1^2}]$ and $s_h[s_{1^2}]$.
\end{remark}

\begin{example}
We have
\[
	s_{(3,1)}[s_{1^2}] = s_{(21^6)}+s_{(2^21^4)}+s_{(2^31^2)}+s_{(321^3)}+s_{(32^21)}+s_{(3^21^2)}+s_{(431)}.
\]
Note that $w(431)=(21^6)$, $w(3^21^2)=(2^21^4)$, $w(32^21)=(2^31^2)$, and $w(321^3)=(321^3)$ and indeed the coefficients of $s_\gamma$ and 
$s_{w(\gamma)}$ match.
\end{example}

\begin{cor}
We have
\begin{align}
	s_{(2,1^{h-2})}[s_{1^2}] &= \sum_{\mu \in \mathcal{T}_{2h}} s_\mu + \sum_{\nu \vdash 2h \text{ threshold}} \left\lfloor \frac{b_{\nu}-1}{2} \right\rfloor s_\nu,\\
	s_{(2,1^{h-2})}[s_{2}] &= \sum_{\mu \in \mathcal{T}_{2h}} s_{\mu'} + \sum_{\nu \vdash 2h \text{ threshold}} \left\lfloor \frac{b_{\nu}-1}{2} \right\rfloor s_{\nu'},
\end{align}
where $\mathcal{T}_{2h}$ is the set of all partitions $\lambda$ of $2h$ with $\lambda_i'=\lambda_i+1$ for all $i$ except either
\begin{enumerate}[label=(\roman*)]
\item two distinct $i$ in the range
$1\leqslant i \leqslant d(\lambda)$ for which either $\lambda_i'=\lambda_i+2$
and the cell $(\lambda_i',i)$
is a corner or $\lambda_i'=\lambda_i$; or
\item one $i$ in the range $1\leqslant i \leqslant d(\lambda)$ with $\lambda_i'=\lambda_i+3$.
\end{enumerate}
\end{cor}

\begin{proof}
The proof follows the same outline as the proof of Corollary~\ref{cor.small hook} using~\eqref{equation.small hook} with $r=h-1$ (instead of $r=1$) so that
$s_{(2,1^{h-2})}[s_{1^2}] = s_{1^2} s_{1^{h-1}}[s_{1^2}] - s_{1^h}[s_{1^2}]$.
\end{proof}

\begin{cor}
We have for $h \geqslant 3$
\[
\begin{split}
	s_{(h-2,1,1)}[s_{1^2}] &= \sum_{\mu\vdash 2h} a_\mu s_\mu,\\
	s_{(h-2,1,1)}[s_{2}] &= \sum_{\mu\vdash 2h} a_\mu s_{\mu'},
\end{split}
\]
where
\[
	a_\mu = \begin{cases}
	\binom{b_\mu-1}{2} & \text{if $\mu$ is even,}\\
	\left(\sum_{\nu \vdash 2(h-2) \text{ even}} c^\mu_{\nu (2,1,1)} \right)-1  & \text{if $\mu \in \mathcal{P}_{2h}$,}\\
	\sum_{\nu \vdash 2(h-2) \text{ even}} c^\mu_{\nu (2,1,1)}   & \text{otherwise.}
	\end{cases}
\]
\end{cor}

\begin{proof}
From Equation \eqref{equation.small hook} with $r=2$ we know that
$s_{(h-2,1,1)}[s_{1^2}] = s_{h-2}[s_{1^2}] s_{1^2}[s_{1^2}] - s_{(h-1,1)}[s_{1^2}]$.
We use the fact that $s_{1^2}[s_{1^2}]=s_{(2,1,1)}$ and the expression for
$s_{(h-1,1)}[s_{1^2}]$ from Corollary~\ref{cor.small hook}.
In the Schur expansion of $s_{h-2}[s_{1^2}]$
only even partitions appear.  The partitions $\mu$ indexing the Schur
functions in the product
$s_{h-2}[s_{1^2}] s_{1^2}[s_{1^2}] = s_{h-2}[s_{1^2}] s_{(2,1,1)}$ fall into four cases:
\begin{itemize}
\item Case 1: $\mu$ has exactly 4 odd columns
\item Case 2: $\mu$ has exactly two equal odd columns
\item Case 3: $\mu$ has exactly two distinct odd columns
\item Case 4: $\mu$ is an even partition.
\end{itemize}
Subtracting the combinatorial formula for $s_{(h-1,1)}[s_{1^2}]$ given by Corollary~\ref{cor.small hook} does not subtract anything in Cases 1 and 2, subtracts 1
in Case 3, and $b_\mu-1$ in Case 4. Using that the Littlewood--Richardson coefficient $c^\mu_{\nu (2,1,1)}$ for all $\nu \subseteq \mu$ with 
$\nu \vdash 2(h-2)$ even count the coefficients of $s_\mu$ in $s_{h-2}[s_{1^2}] s_{(2,1,1)}$ hence proves the last two cases in the formula for $a_\mu$.
To prove the first case in $a_\mu$, note that $\sum_{\nu \vdash 2(h-2) \text{ even}} c^\mu_{\nu (2,1,1)}$ for $\mu$ even is equal to
$\binom{b_\mu}{2}$ since this number is counted by Littlewood--Richardson tableaux of shape $\mu/\nu$ and weight $(2,1,1)$.
Recall that a Littlewood--Richardson tableau is a semistandard Young tableau
such that the row reading of the tableau is a reverse lattice word. For $\mu$ and $\nu$ even
these tableaux have to have $21$ in the rightmost column of $\mu/\nu$ and $31$ in the leftmost column 
of $\mu/\nu$. This means we need to pick two corners of $\mu$, which can be done in $\binom{b_\mu}{2}$ ways. Noting
$\binom{b_\mu}{2} - (b_\mu-1) = \binom{b_\mu-1}{2}$ proves the first equation in the corollary.

The second equation follows again from~\eqref{equation.s transpose}.
\end{proof}

For general hooks, we have the following result.

\begin{cor}
\label{cor.inclusion exclusion}
For $h>k\geqslant 0$, we have
\begin{equation}
\label{equation.hook inclusion exclusion}
	s_{(h-k,1^k)}[s_{1^2}] = \sum_{i=0}^k \sum_{\substack{\mu \vdash 2h\\ \nu \vdash 2(h-k+i) \text{ even}\\ \rho \vdash 2(k-i) \text{ threshold}}}
	(-1)^i c^\mu_{\nu\rho} s_\mu.
\end{equation}
\end{cor}

\begin{proof}
We prove this result by induction on $k$. For $k=0$, Equation~\eqref{equation.hook inclusion exclusion} reads
\[
	s_{h}[s_{1^2}] = \sum_{\substack{\mu \vdash 2h\\ \nu \vdash 2h \text{ even}\\ \rho \vdash 0 \text{ threshold}}}
	 c^\mu_{\nu \rho} s_\mu = \sum_{\mu \vdash 2h \text{ even}} s_\mu,
\]
which is true by~\eqref{equation.row plethysm}.

Now assume by induction that~\eqref{equation.hook inclusion exclusion} holds for $k-1$. Note that by~\eqref{equation.small hook}
\[
	s_{(h-k,1^k)}[s_{1^2}] = s_{h-k}[s_{1^2}] s_{1^k}[s_{1^2}] - s_{(h-k+1,1^{k-1})}[s_{1^2}].
\]
By Theorem~\ref{theorem.row col}  and the Littlewood--Richardson rule, the term $s_{h-k}[s_{1^2}] s_{1^k}[s_{1^2}]$ is the term $i=0$
in~\eqref{equation.hook inclusion exclusion}. By induction, the second term equals
\[
\begin{split}
	s_{(h-k+1,1^{k-1})}[s_{1^2}] &= \sum_{i=0}^{k-1} \sum_{\substack{\mu \vdash 2h\\ \nu \vdash 2(h-k+1+i) \text{ even}\\ \rho \vdash 2(k-1-i) \text{ threshold}}}
	(-1)^i c^\mu_{\nu\rho} s_\mu\\
	&=  - \sum_{i=1}^k \sum_{\substack{\mu \vdash 2h\\ \nu \vdash 2(h-k+i) \text{ even}\\ \rho \vdash 2(k-i) \text{ threshold}}}
	(-1)^i c^\mu_{\nu\rho} s_\mu,
\end{split}
\]
which are the remaining terms in~\eqref{equation.hook inclusion exclusion}, proving the claim.
\end{proof}

\begin{remark}
It can be deduced from Corollary~\ref{cor.inclusion exclusion} that the coefficient of $s_\mu$ with $\mu$ even in $s_{(h-3,1,1,1)}[s_{1^2}]$ is
$(b_\mu-1)(b_\mu-2)(2b_\mu-3)/6$.  The coefficient of $s_\mu$ with $\mu$ even in
$s_{(h-4,1,1,1,1)}[s_{1^2}]$ is not dependent solely on $b_\mu$ since the coefficient of $s_{(3^22^21^4)}$
in $s_{(3,1,1,1,1)}[s_{1^2}]$ is 1 while the coefficient of $s_{(4^22^21^2)}$ is $2$.
\end{remark}

\bibliographystyle{plain}
\bibliography{references}

\end{document}